\documentclass[12pt,a4paper]{amsart}
\usepackage{amsmath,amscd}
\usepackage{latexsym,amsmath,amssymb,mathrsfs,setspace,enumerate}
\usepackage{subcaption}
\usepackage[T1]{fontenc}
\usepackage[utf8]{inputenc}
\usepackage{lmodern}
\usepackage{amssymb}
\newtheorem{lemma}{Lemma}
\newtheorem{prop}{Proposition}
\newtheorem{thm}{Theorem}

\newtheorem{exam}{Example}
\newtheorem{cor}{Corollary}
\newtheorem{dfn}{Definition}

\title{\textbf{Local Subsemigroups and Variants of\\Some class of Semigroups }}
\author {Siji Michael$^1$ and P. G. Romeo$^2$}
\address{$^{1}$ Research Scholar, Department of Mathematics, Cochin University of Science and Technology, Kochi, Kerala, INDIA. ,$^2$ Professor, Dept. of Mathematics, Cochin University of Science and Technology, Kochi, Kerala, INDIA.  }
\email{$romeo_-parackal@yahoo.com,\, sijimichael999@gmail.com $}
\thanks{Second author wishes to thank Cochin University of Science And technology for providing financial support  under University JRF Scheme.}
\usepackage{graphicx}
\graphicspath{ {./images/} }
\begin{document}
\begin{abstract}
 For an element $ a $ in a semigroup $ S $ the local subsemigroup of $ S $ with respect to $ a $ is the subsemigroup $ aSa $ of $ S $  and a semigroup $(S,\star_a)$ where $\star_a$ is the sandwich operation $x\star_a y=xay$ for all $x,y\in S$ is a variant of $S$. In this paper we discuss the structures of local subsemigroups of full transformation semigroups and symmetric inverse monoids. It is also shown that the set of all local subsemigroups of finite symmetric inverse monoids and the set of all variants of all finite symmetric inverse monoids are same upto isomorphism.
\end{abstract}
\maketitle

Local subsemigroups and semigroup variants are two well known constructions in semigroups. In \cite{East}, James East studied the link between these two and it is shown that in the case of full transformation semigroup on a set $X$ the two constructions lead exactly to the same class of semigroups up to isomorphism. In this paper we discuss about the structure of local subsemigroups of finite full transformation semigroups and symmetric inverse monoids. The structures studies are carried out using the egg box diagrams (\cite{Cli}, \cite{How}) obtained with the semigroups package \cite{mit} for GAP \cite{gap}.

\section{PRELIMINARIES}

In the following we briefly recall some basic notions and results concerning finite transformation semigroups and symmetric inverse monoids. A semigroup $ S $ is a nonempty set together with an associative binary operation. An element $  x\in S $ is regular if $ xyx=x $ and $ yxy=y $ for some $ y\in S $ and a semigroup $ S $ is called regular if all elements of $ S $ are regular. An element $ x\in S $ is called an idempotent if $ x^{2}=x $ the collection of all idempotents in $S$ will be denoted by $ E(S) $. Two elements of a semigroup $ S $ are said to be $ \mathscr{L,R,J} $-equivalent if they generate the same principal left,right,two sided ideals respectively.

The join of the equivalence relations $ \mathscr{L} $ and $ \mathscr{R} $ is denoted by $ \mathscr{D} $ and their intersection by $ \mathscr{H} $. These equivalence relations are of fundamental importance in the study of the structure of a semigroups are introduced by J.A.Green and are known as Green's relations. The egg-box diagram visualizes $ \mathscr{D} $-class structure of semigroup $ S $ using rectangular patterns. In each rectangular pattern (which corresponds to each $ \mathscr{D} $-class) the rows corresponds to the $ \mathscr{R} $-classes and the columns to $ \mathscr{L} $-classes contained in a $ \mathscr{D} $-class.

 For a finite set $ X $ with $ \lvert X \rvert = n $ the set of all transformations of $ X $ 
 (ie.,  all functions $ X\mapsto X $) under the operation of composition of maps is a  the full transformation semigroup on $ X $ and is denoted as $ T_{X} $. It is well known that $ T_{X} $ is a  regular semigroup.  For $ f\in T_{X} $ the image and rank of $f$ will be denoted by $$ im(f) =\{f(x): x\in X \} $$ 
 $$rank(f)= \lvert im(f) \rvert .$$
 
 Symmetric inverse monoid on a finite set $ X $ is the set of all partial bijections on $ X $ 
 (ie., all bijections from a subset of $ X $ to a subset of $ X $) with composition of maps as the binary operation and is written as $ IS_{X}$ . The domain and range of a partial permutation $ \alpha $ is denoted as $ dom  \alpha $ and $ ran \alpha $ respectively. We denote the rank of empty partial permutation as zero. Idempotents of $ IS_{X} $ are the identity mappings on each of the subsets of $ X $. ie., 
 $$ E(IS_{X})= \{1_{A} : A \subseteq  X \} $$.
 
 \begin{dfn}(\cite{East})
 	Let $ S $ be a semigroup, and $ a $ an element of $ S $. The set $ aSa =  \{ axa : x\in S\} $ is a subsemigroup of $ S $ called local subsemigroup of $ S $ with respect to $ a $.
 \end{dfn}
 
 \begin{dfn} ( \cite{Hick})
 	Let $ S $ be a semigroup and $ a $ be an element of $ S $. An associative sandwich operation $ \star_{a} $ can be defined on $ S $ by $ x\star_{a}y = xay $ for all $ x,y\in S $.The semigroup $ (S,\star_{a}) $ is called the variant of $ S $ with respect to $ a $ and is denoted as $ S^{a} $.
 \end{dfn}
For a semigroup $ S $ if $ a \in S^{\ast} $ then the variant of $ S $ with respect to $ a $, $ S^{a} \cong S $ \cite{Maz}. The variants of full transformation semigroups, semigroup of binary relations, etc. are widely studied \cite{Dolinka,Tsy}.

\section{Local Subsemigroups Of Full Transformation Semigroups}
This section we discuss the local subsemigroups of full transformation semigroup on a finite set $X$.
\begin{prop}
	Let $X$ be the set with $ \lvert X \rvert = n $. For $ \alpha \in T_{X} $ with $ rank(\alpha) = m  \leq n $. Then $ \alpha T_{X}\alpha $ is a local subsemigroup of $ T_X $ with respect to $ \alpha $ and its order is 
	$$ \lvert \alpha T_{X}\alpha \rvert=  \lvert T_{m}\rvert .$$
\end{prop}	
\begin{proof}
	Let $ \alpha \in T_{X} $ define an equivalence relation $\pi_{\alpha} $ such that for $ x,y \in X $, $ x \pi_{\alpha} y $ if $ x\alpha = y\alpha $. Then the equivalence classes is a partitions $ dom\alpha $ and $ \lvert \dfrac{X}{\pi_{\alpha}} \rvert = rank(\alpha)$, also 
	$ \pi_{\alpha} \subseteq \pi_{\alpha \beta \alpha} $. Since $ \beta $ varies between all elements of $ T_{X} $, we get $ m $ choices for each partition in $ dom\alpha $ and hence will have $ m^{m} $ elements.
	
\end{proof}
\begin{cor}
	For $ \alpha \in T_{X} $ where $ \lvert X \rvert = n $, with $ rank( \alpha )= n $ then $ \alpha T_{X}\alpha $ is same as $ T_{X} $.
\end{cor}
\begin{proof}
	It is clear that $ \alpha T_{X}\alpha \subseteq T_{X}$. By the above proposition, we get $ \lvert \alpha T_{X} \alpha \rvert $ is same as $ \lvert T_{X}\rvert $. Hence $ \alpha T_{X}\alpha $ is same as $ T_{X} $.
\end{proof}
Comparing egg box diagrams of local subsemigroups we have observed that for each rank $ m $ there are only two structures obtained: one is a full transformation semigroup of order $  m $ and another structure is a variant of full transformation semigrop of order $ m $. From \cite{East} we get that any local subsemigroup of finite full transformation semigroup is isomorphic to a variant of a finite full transformation semigroup (may be different). The following theorem states the same.
\begin{thm}\label{thm1}
	Let $ n $ be a positive integer,and let $ a \in T_{n} $ with $ rank(a)=r $. Then
	\begin{enumerate}
	 \item $ aT_{n}a =T_{r}^{c} $ for some $ c \in T_{r} $ with $ rank(c) = rank(a^{2}) $.
	\end{enumerate}
\end{thm}

It is found that the local subsemigroups of full transformation semigroups can be classified using stabiliser and stable image of transformations. 
\begin{dfn}\cite{Ara}
	For $ \alpha \in T_{X} $ we can define the stable image of $ \alpha $ denoted as $ sim(\alpha) $ by 
	$$ sim(\alpha)	= \{ x\in X : x \in im(\alpha^n) \quad for \quad every \quad n\geq 0 \}. $$
\end{dfn}	
\begin{dfn}\cite{Ara}
	For $ \alpha \in T_{X} $ we define the stabiliser of $ \alpha $ as the smallest positive integer $ s\geq 0 $ such that im($ \alpha^{s} $) = im($ \alpha^{s+1} $).
\end{dfn}
From the above two definitions it is clear that if $ \alpha $ has the stabiliser $ s $ then $ sim(\alpha)= im(\alpha^{s}) $. By comparing the egg-box diagrams of local subsemigroups of finite full transformation semigroups we obtained the following results.
\begin{prop}\label{pro2}
	Let $ \alpha \in T_{n} $ with rank($ \alpha $)= $ m\leq n $, and stabiliser of $ \alpha $ is 1 then $ \alpha T_{n}\alpha $ is isomorphic to $ T_{m} $.
\end{prop}
\begin{proof}
  From Theorem \ref{thm1}, it follows that $ \alpha T_{n}\alpha \cong T_{m}^{c} $ with $ rank(c)= rank(\alpha^{2}) $ . Since stabiliser of $ \alpha $ is $ 1 $, $ rank(\alpha^{2}) = rank(\alpha) = m $. $ c $ being a permutation in $ T_{m} $, $ T_{m}^{c} \cong T_{m} $. Hence, $ \alpha T_{n}\alpha \cong T_{m} $.
	
\end{proof}
\begin{exam}
	Consider transformation $ \alpha = (2432) \in T_{4} $. Then $ rank(\alpha) $ is $ 3 $ and the stabiliser of $ \alpha $ is $ 1 $. Then by Proposition \ref{pro2}, local subsemigroup of $ \alpha $ is isomorphic to $ T_{3} $ (see figure \ref{fig1}).
\end{exam}

\begin{figure}[h]
	\includegraphics[width=3cm, height=10cm]{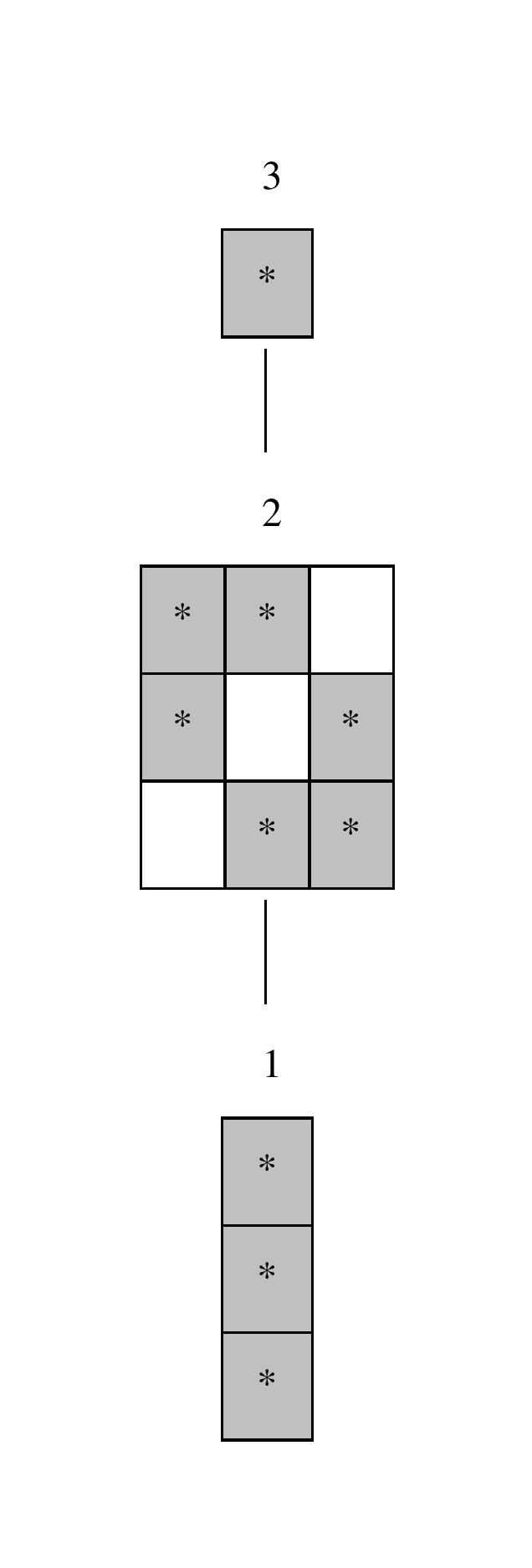}	
	\caption{Egg-box diagram of local subsemigroup when $ \alpha $ = (2432)}
	\label{fig1}
\end{figure}
\begin{prop}
	Let $ \alpha, \beta \in T_{n} $ with rank ($ \alpha $)= rank($ \beta $) and the stabilisers of $ \alpha $ and $ \beta $ are not $1$ then the local subsemigroups of $ \alpha $ and $ \beta $ are isomorphic.
\end{prop}
The above proposition is illustrated in the following example.
\begin{exam}
	Consider transformations $  \beta, \gamma \in T_{4} $ such that  $ \beta=(2343) $ and $ \gamma=(1123) $. Then $ \beta, \gamma $ are of rank $ 3 $ and the stabilisers are $ 2 $ and $ 3 $ respecively (see figure \ref{fig2}).
\end{exam}

\newpage

\begin{figure}[h]
	\begin{subfigure}{0.5\textwidth}
		\centering
		\includegraphics[width=0.8\linewidth]{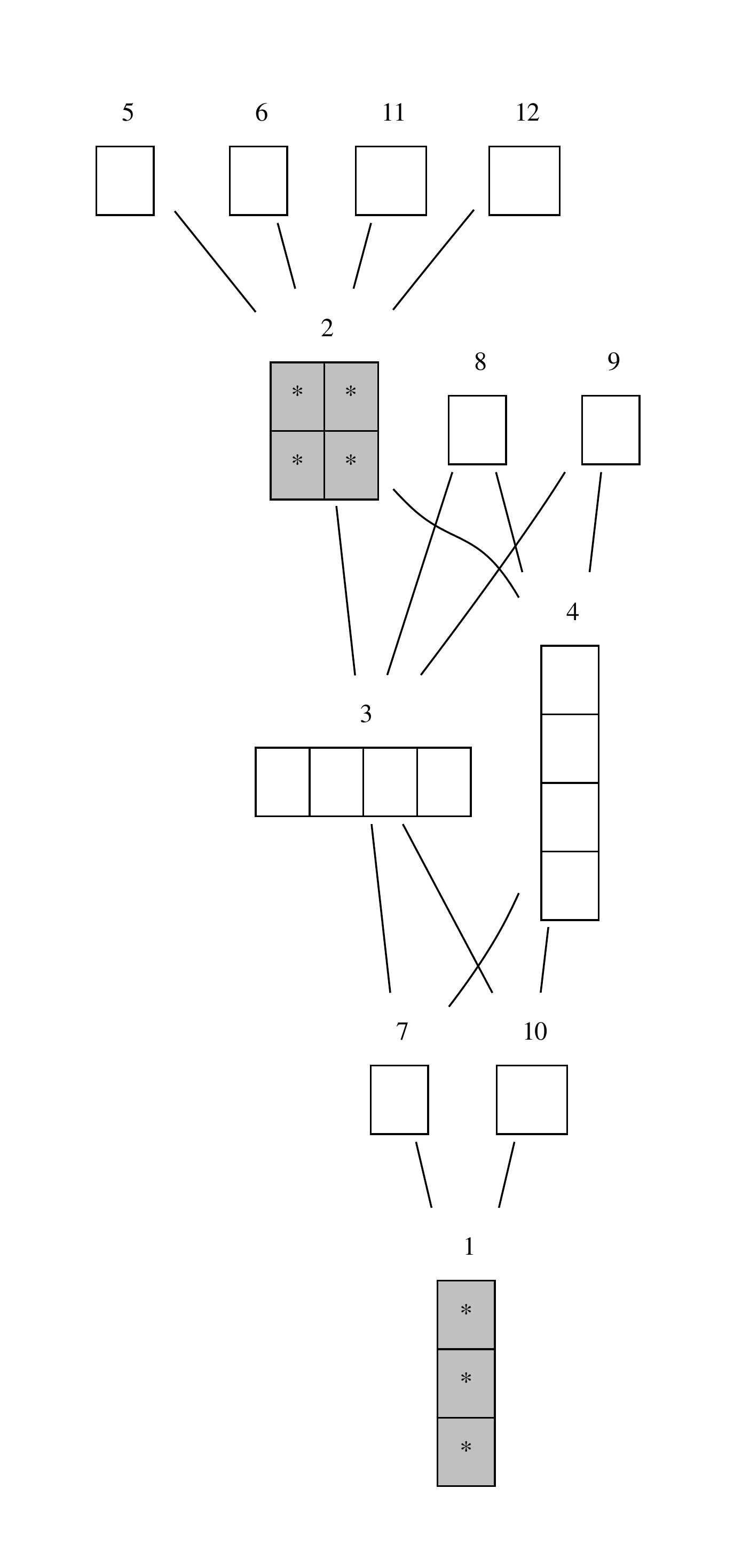}
	\end{subfigure}%
   \begin{subfigure}{0.5\textwidth}
    	\centering            
    	\includegraphics[width=0.8\linewidth]{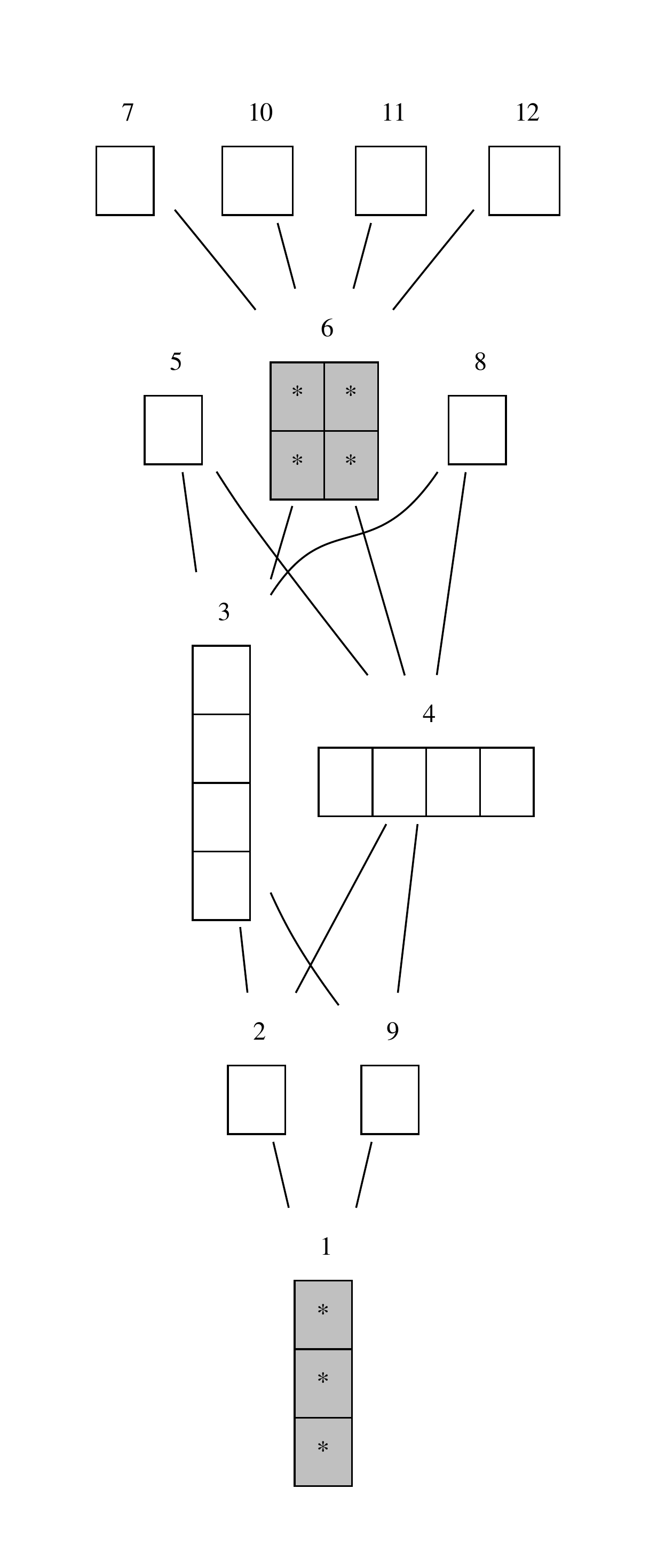}	
    \end{subfigure}
	\caption{Egg-box diagrams of local subsemigroups when $ \beta $ = (2343) and $ \gamma $ = (1123)}
	\label{fig2}
\end{figure}

\section{Local Subsemigroups and Variants}
In the following we compare the structures of local subsemigroups of symmetric inverse monoids. From here onwards we denote $ X = \{1,2,\cdots,n\} $ and $ IS_{n} $ denote symmetric inverse monoid on $ X $.  
\begin{prop}
	 For $ \alpha \in IS_{n} $ with rank($ \alpha )=m\,\, \lvert \alpha IS_{n}\alpha \rvert $ equals $\lvert IS_{m} \rvert = \sum_{k=0}^{m} {m\choose k}^2 k! $.
\end{prop}
\begin{proof}
	Let $ \beta \in IS_{n} $, then
	$$ dom(\alpha\beta \alpha)\subseteq dom(\alpha)      $$
	and $ ran(\alpha \beta \alpha) \subseteq ran(\alpha) $. That implies $ rank(\alpha\beta\alpha) \leq rank(\alpha) $.

	Therefore all  the elements of $ \alpha IS_{n} \alpha $ will be of rank $ \leq $ rank($ \alpha $). Then there will be $ {m\choose k} $ different choices for $ dom(\alpha\beta\alpha)  $ with rank $ k $ and $ {m\choose k} $ different choices for $ ran(\alpha\beta\alpha) $. For each domain and range, there will be $ k! $ different bijections. Hence we have $ {m\choose k}{m\choose k} k! $ bijections of rank $ k $. Since the rank can be varied from $ 0 $ to $ m $, we get $ \alpha IS_{X}\alpha $ contains $ \sum_{k=0}^{m} {m\choose k}^2 k! $ elements.
\end{proof}
\begin{prop}
	For $ \alpha \in IS_{n} $ with  rank($\alpha$ ) $=n $ local subsemigroup $ \alpha IS_{n}\alpha$ is isomorphic to $ IS_{n} $.
\end{prop}
\begin{proof}
	Clearly $ \alpha IS_{n}\alpha \subseteq IS_{n} $.
	
	For the reverse inclusion, let $ \beta\in IS_{n} $. Since $ rank(\alpha)=n $, $ \alpha\in S_{n} $. Therefore there exists $ \alpha^{-1} $ in $ IS_{n} $ such that $ \alpha^{-1}\beta\alpha^{-1} \in IS_{n} $  which implies $ \alpha (\alpha^{-1}\beta\alpha^{-1})\alpha = \beta \in \alpha IS_{n} \alpha $.
	
	Therefore, $ IS_{n}\subseteq \alpha IS_{n}\alpha $.
	Hence the proof.	
\end{proof}
\begin{prop}\label{pr1}
	For $ \alpha\in IS_{n} $ with $ rank(\alpha)< n $ and $ rank(\alpha^{2}) = rank(\alpha) $, then $ \alpha IS_{X}\alpha $ is isomorphic to $ IS_{A} $, where $ A= ran(\alpha) $.
\end{prop}
\begin{proof}
	Let $ \alpha\in IS_{n} $ with $ rank(\alpha)< n $, and $ \alpha $ be a permutation on a subset $ A $ of $ X $.
	
	For $ \beta \in IS_{n} $, $ dom(\alpha \beta\alpha) \subseteq A $ and $ ran(\alpha\beta\alpha) \subseteq A $ which implies $ \alpha\beta\alpha \in IS_{A} $.
	Therefore, $ \alpha IS_{n} \alpha \subseteq IS_{A} $.
	By result 2, they have the same number of elements. Hence $ \alpha IS_{n} \alpha $ is isomorphic to $ IS_{A} $.	
\end{proof}

Below, we describe the relation between local subsemigroups and variants of finte symmetric inverse monoids. The following theorem states the main result in this regard.
\begin{thm}\label{thm}
	Let $ n $ be a positive integer and let $ \alpha\in IS_{n} $, with $ rank(\alpha)=r $.Then
	\begin{enumerate}
		\item $ \alpha IS_{n}\alpha \cong IS_{r}^{c}  $ for some $ c\in IS_{r} $ with $ rank(c)=rank(\alpha^2) $.
		\item $ IS_{n}^{\alpha} \cong \beta IS_{2n-r}\beta $ for some $ \beta \in IS_{2n-r} $, $ rank(\beta) =n $ and $ rank(\beta^2)= r $.
	\end{enumerate}
\end{thm}
Before proving theorem some results of variants of semigroups are recalled below \cite{East}.
\begin{lemma}\label{lm1}
	Let $ a $ and $ b $ be regular elements of a semigroup $ S $ and define the idempotents $ e=ab $ and $ f=ba $. Then $ aSb=eSe $ and $ bSa=fSf $.
\end{lemma}
\begin{proof}
	It is clear that $ eSe = abSab \subseteq aSb $. Let $ x \in aSb $, then $ x = aub $ where $ u\in S $. Then $ x = abaubab = ab(aub)ab \in abSab = eSe $. Hence, $ aSb=eSe $. Similarly the other part follows.

\end{proof}
\begin{lemma}\label{lm2}
	If $ a $ and $ b $ are elements of a semigroup $ S $ satisfying $ a=aba $ and $ b=bab $ then
	$$ (aSa,\cdot) \cong (aSb,\star_{aab}) \cong (bSa,\star_{baa})  $$
\end{lemma}
\begin{lemma}\label{lm3}
	If $ \phi:S \to T $ is a semigroup isomorphism and if $ c\in S $, then $ S^{c} \cong T^{\phi(c)}  $.
\end{lemma}
\begin{proof}
	We have for $ a,b\in S $, $\phi(ab)=\phi(a)\phi(b)  $.
	Now,$ \phi(a\star_{c} b) = \phi(acb) = \phi(a) \phi(c) \phi(b) = \phi(a)\star_{\phi(c)} \phi(b) $. Hence the result follows.
\end{proof}
Proof of Theorem \ref{thm}.\begin{proof}
Let $ n $ be a positive integer and fix some $ \alpha \in IS_{n} $ with $ rank(\alpha) = r $. Also write $ X=\{1,2,..,n\} $, $ Y=\{1,2,..,r\} $ and $ Z=\{1,2,..2n-r\} $. Re-labeling if necessary we assume $ ran\alpha = Y $ and  we can write $\alpha = \bigl(\begin{smallmatrix}
x_{i} \\
i
\end{smallmatrix}\bigr)$, where $ x_{i}\in X ,i=1,..r $.
\begin{enumerate}
	\item Let $ \beta $ be the unique inverse of $ \alpha $ in $ IS_{X} $.Then $ e= \alpha\beta = 1_{dom\alpha} $. $ \alpha = \alpha\beta\alpha $ implies $ rank(\alpha^{2})= rank(\alpha^{2}\beta) $. Now, by Lemma \ref{lm2}, $ (\alpha IS_{n}\alpha,\cdot) \cong (\alpha IS_{n}\beta,\star_{\alpha\alpha\beta})$. Also by Lemma \ref{lm1}, $ \alpha IS_{n}\beta = eIS_{n}e $. And by Proposition \ref{pr1}, $ (\alpha IS_{n}\beta,\cdot) \cong (e IS_{n}e,\cdot) \cong (IS_{Y},\cdot) $. Hence we get, $ \alpha IS_{n}\alpha = (\alpha IS_{n}\alpha, \cdot)\cong (\alpha IS_{n}\beta, \star_{\alpha\alpha\beta}) \cong (IS_{Y}, \star_{c}) = IS_{r}^{c} $ where $ c=\alpha\alpha\beta|_Y $. Also we get $ rank(c)=rank(\alpha^{2}) $.
	\item Define partial Transformations $ \beta,\gamma$ on $ Z $ by  $$ \beta = \bigl(\begin{smallmatrix}
	x_{i} & y_{j}\\
	i & j
	\end{smallmatrix}\bigr) \quad \gamma = \bigl(\begin{smallmatrix}
	i & j \\
	x_{i} & y_{j}
	\end{smallmatrix}\bigr) $$, where $ i=1,..r , j=n+1,..2n-r $ and $ y_{j}\in X\backslash dom\alpha $.
	
	Then we get, $ \beta^{2}\gamma = \beta^{2} = \bigl(\begin{smallmatrix}
	x_{i} \\
	i
	\end{smallmatrix}\bigr)_{i=1,..r}  $.
	
	Also, we may varify that $ \beta\gamma\beta = \beta $ and $ \gamma\beta\gamma = \gamma $. Then by Lemma \ref{lm2}, $ (\beta IS_{Z}\beta,\cdot) \cong (\beta IS_{Z}\gamma,\star_{\beta\beta\gamma})$. Now we have, $ e= \beta\gamma = 1_{X} $. By Lemma \ref{lm1} and  Proposition \ref{pr1} , $ (\beta IS_{X} \gamma, \cdot) = (1_{X}IS_{X}1_{X}, \cdot) = (IS_{X}, \cdot) $. Also by lemma \ref{lm3} and since $ \beta\beta\gamma |_X = \alpha  $ we get $ (\beta IS_{Z}\gamma, \star_{\beta\beta\gamma}) \cong (IS_{X}, \star_{\alpha}) $. From these we get,$ IS_{n}^{\alpha} \cong (\beta IS_{Z}\gamma, \star_{\beta\beta\gamma}) \cong (\beta IS_{Z}\beta, \cdot) = \beta IS_{2n-r}\beta $ where $ rank(\beta)= n $ and $ rank(\beta^{2})= r $.
\end{enumerate}
\end{proof}
\begin{prop}
	For $ \alpha, \beta \in IS_{X} $ with same rank such that $ rank(\alpha^{2}) = rank(\beta^{2}) $ then the local subsemigroups of $ \alpha $ and $ \beta $ are isomorphic.
\end{prop}
\begin{proof}
    Let $ rank(\alpha) = rank(\beta) = r $. Then by Theorem \ref{thm},	$ \alpha IS_{n}\alpha \cong IS_{r}^{c}  $ for some $ c\in IS_{r} $ with $ rank(c) = rank(\alpha^2) $. Similarly we get $ \beta IS_{n}\beta \cong IS_{r}^{d}  $ for some $ d\in IS_{r} $ with $ rank(d)=rank(\beta^2) $. Theorem $ 1.1 $ of \cite{Tsy1} states that $ IS^{c} \cong IS^{d} $ if $ rank(c)=rank(d) $. Since $ rank(\alpha^{2}) = rank(\beta^{2}) $, we get local subsemigroups of $ \alpha $ and $ \beta $ are isomorphic.
\end{proof}

{15}

\end{document}